\documentclass[3p, numbers]{elsarticle}
\usepackage{kpfonts} 
\usepackage{amsthm}
\usepackage{mathtools}
\usepackage{bbm}
\usepackage{tikz-cd}
\usepackage{hyperref}

\usepackage{amssymb}
\usepackage{amsmath}
\usepackage{amsfonts}
\usepackage{amsthm}
\usepackage{float}
\usepackage{times}
\usepackage{color}
\usepackage{proof}

\newtheorem{Theorem}{Theorem}[section]
\newtheorem{Proposition}[Theorem]{Proposition}
\newtheorem{Lemma}[Theorem]{Lemma}

\newtheorem{Corollary}[Theorem]{Corollary}

\newcommand{\f}{\ensuremath{\varphi}}
\newcommand{\ps}{\ensuremath{\psi}}
\newcommand{\x}{\ensuremath{\chi}}
\newcommand{\la}{\lambda}

\newcommand{\ga}{\gamma}
\newcommand{\ut}{1}
\newcommand{\nf}{0}
\newcommand{\m}{\mathbf} 

\newcommand{\N}{\mathbb{N}}

\newcommand{\si}{\sigma}
\newcommand{\De}{\Delta}

\newcommand{\Si}{\Sigma}

\newcommand{\lang}{{\mathcal L}}
\newcommand{\fm}{\mathrm{Fm}}
\newcommand{\fma}{\mathbf{Fm}}
\newcommand{\vty}[1]{\mathcal{#1}}
\newcommand{\op}[1]{\mathbb{#1}}
\newcommand{\lgc}[1]{\mathrm{#1}}
\newcommand{\mdl}[1]{\models_{\lgc{#1}}}
\newcommand{\der}[1]{\vdash_{\lgc{#1}}}\newcommand{\eq}{\approx}

\newcommand*{\InCRL}{\mathcal{I}{\rm n}\mathcal{CRL}}

\renewcommand{\iff}{\enspace \Longleftrightarrow \enspace }    
\newcommand{\defiff}{\enspace :\Longleftrightarrow \enspace }

\begin{document}
\begin{frontmatter}
\journal{ArXiv}
\title{Theorems of Alternatives for Substructural Logics}
\author[ac]{Almudena~Colacito}
\address[ac]{Mathematical Institute, University of Bern\\Sidlerstrasse 5, 3012 Bern, Switzerland}
\ead{almudena.colacito@math.unibe.ch}
\author[ng]{Nikolaos~Galatos}
\ead{ngalatos@du.edu}
\author[ac]{George~Metcalfe}
\ead{george.metcalfe@math.unibe.ch}
\address[ng]{Department of Mathematics, University of Denver\\2360 S. Gaylord St., Denver, CO 80208, USA}

\begin{abstract}A theorem of alternatives provides a reduction of validity in a substructural logic to validity in its  multiplicative fragment. Notable examples include a theorem of Arnon Avron that reduces the validity of a disjunction of multiplicative formulas in the ``R-mingle'' logic $\lgc{RM}$ to the validity of a linear combination of these formulas, and Gordan's theorem for solutions of linear systems over the real numbers, that yields an analogous reduction for validity in Abelian logic $\lgc{A}$. In this paper, general conditions are  provided for axiomatic extensions of involutive uninorm logic without additive constants to admit a theorem of alternatives. It is also shown that a theorem of alternatives for a logic can be used to establish (uniform) deductive interpolation and completeness with respect to a class of dense totally ordered residuated lattices.
\end{abstract}


\end{frontmatter}
\section{Introduction}\label{sec:intro}

In~\cite{Avr87} Arnon Avron proved a remarkable theorem relating derivability in the ``R-mingle" logic $\lgc{RM}$ (see, e.g.,~\cite{Dun70,AB75,Avr86}) formulated with connectives $+$, $\lnot$, $\land$, and $\lor$, to derivability in its multiplicative fragment  with connectives $+$ and $\lnot$. More precisely, Avron proved that a disjunction of multiplicative formulas $\f_1 \lor\ldots\lor \f_n$  is derivable in $\lgc{RM}$ if and only if $\f_{j_1} + \ldots + \f_{j_k}$ is derivable in $\lgc{RM}$  for some $1 \le j_1 < \ldots  <  j_k \le n$. Indeed, two proofs are given of this result. The first is an easy consequence of a (quite hard) cut elimination proof for a proof system for $\lgc{RM}$ defined in the framework of hypersequents, introduced in the paper as sequences of sequents. The second proof is semantic and makes use of the completeness of  $\lgc{RM}$ and its multiplicative fragment with respect to an infinite-valued and a three-valued matrix, respectively.

The central aim of the first part of this paper is to show that Avron's theorem belongs to a family of results that may be understood as ``theorems of alternatives'' for substructural logics. Such theorems in the field of linear programming are duality principles  stating that either one or another linear system has a solution over the real numbers, but not both  (see, e.g.,~\cite{Dan63}). In particular,  Gordan's theorem (replacing real numbers with integers) asserts that for any $M \in \mathbb{Z}^{m \times n}$,
\begin{center}
{\em either} \ 
$y^T M > \mathbf{0}$ \,for some $y \in \mathbb{Z}^m$ \ {\em or} \ $Mx = \mathbf{0}$ \,for some $x \in \N^n {\setminus} \{\mathbf{0}\}$.
\end{center}

\noindent
This version of Gordan's theorem is proved in~\cite{CM17} by extending partial orders on free abelian groups to total orders and formulated as a correspondence between derivability in Abelian logic $\lgc{A}$  (see, e.g.,~\cite{MS89,Cas89,MOG05}) and derivability in its multiplicative fragment. That is, a disjunction of multiplicative formulas $\f_1 \lor\ldots\lor \f_n$  is derivable in $\lgc{A}$ if and only if  $\la_1 \f_1 +  \cdots + \la_n \f_n$ is derivable in $\lgc{A}$ for some $\la_1,\ldots,\la_n \in \N$ not all $0$. In Section~\ref{sec:alternatives}, we provide a sufficient condition for an axiomatic extension of involutive uninorm logic without additive constants $\lgc{IUL^-}$ (see~\cite{MM07}) to satisfy such an equivalence. The condition is based on derivability and determines a family of substructural logics admitting a theorem of alternatives that includes $\lgc{RM^t}$ ($\lgc{RM}$ with an additional truth constant), involutive uninorm mingle logic without additive constants $\lgc{IUML^-}$ (axiomatized relative to $\lgc{RM^t}$ by $\ut\to\nf$, see~\cite{MM07}), Abelian logic $\lgc{A}$, and the ``balanced" extension $\lgc{BIUL^-}$ of $\lgc{IUL^-}$ with  additional axioms $n\f \to \f^n$ and $\f^n \to n\f$ for each $n\in\N$.

The second part of the paper focuses on applications of  theorems of alternatives. In Section~\ref{sec:interpolation}, we show that if an extension of $\lgc{IUL^-}$ with a theorem of alternatives admits deductive interpolation or right uniform deductive interpolation (see, e.g.,~\cite{MMT14,vGMT17,KM19}) for its multiplicative fragment, then the full logic admits the property. For example, this provides an alternative proof that $\lgc{RM^t}$ admits deductive interpolation (and hence right uniform deductive interpolation), first proved in~\cite{Mey80} (see also~\cite{Avr86,MM12}). In Section~\ref{sec:density}, we show that any extension of $\lgc{IUL^-}$ that derives $\ut\to\nf$ and has a theorem of alternatives  is complete with respect to a class of dense totally ordered residuated lattices. Obtaining such a ``dense chain completeness" result is important in the field of mathematical fuzzy logic as a key intermediate step towards proving that an axiom system is ``standard complete'': that is, complete with respect to a class of algebras with lattice reduct $[0,1]$ (see, e.g.,~\cite{JM02,MM07,CM08,BT16,MT17,GH18}). Although theorems of alternatives hold only for a fairly narrow class of substructural logics, we obtain here new dense chain completeness proofs for $\lgc{IUML^-}$ and $\lgc{A}$, and a dense chain completeness result for $\lgc{BIUL^-}$ that does not seem to be easily proved using other methods developed in the literature.


\section{Preliminaries}\label{sec:prelim}

Let $\lang$ be any propositional language and let $\fm_\lang$ denote the set of formulas of this language over a fixed countably infinite set of variables, denoting arbitrary variables and formulas by $p,q,r,\dots$ and $\f,\ps,\x,\dots$, respectively. Given $\Si \subseteq \fm_\lang$, we let ${\rm Var}(\Si)$ denote the set of variables occurring in $\Si$, shortening ${\rm Var}(\{\f\})$ to  ${\rm Var}(\f)$. We also denote  the formula algebra of $\lang$ by $\fma_\lang$ and recall that a substitution for $\lang$ is a homomorphism $\si \colon \fma_\lang \to \fma_\lang$.

A {\em substitution-invariant consequence relation} over  $\lang$  is a set ${\der{\lgc{L}}} \mathrel{\subseteq}\mathcal{P}(\fm_\lang) \times \fm_\lang$ that satisfies the following conditions for all  $\Sigma \cup \Sigma' \cup\{\f\}\subseteq \fm_\lang$ (writing $\Sigma\der{\lgc{L}}\f$ to denote $\langle\Sigma,\f\rangle\in{\der{\lgc{L}}}$):
\begin{itemize}

\item[\text{(i)}]  if $\f\in \Sigma$, then $\Sigma\der{\lgc{L}} \f$\enspace \emph{(reflexivity)};

\item[\text{(ii)}] if $\Sigma\der{\lgc{L}} \f$ and $\Sigma\subseteq \Sigma'$, then $\Sigma'\der{\lgc{L}} \f$\enspace \emph{(monotonicity)};

\item[\text{(iii)}] if $\Sigma\der{\lgc{L}} \f$ and $\Sigma'\der{\lgc{L}} \ps$ for every $\ps\in \Sigma$, then $\Sigma'\der{\lgc{L}} \ps$\enspace \emph{(cut)};

\item[{\rm (iv)}]  if  $\Sigma\der{\lgc{L}} \f$, then $\sigma[\Sigma]\der{\lgc{L}} \sigma(\f)$ for any substitution $\si$ for $\lang$\enspace \emph{(substitution-invariance)}.

\end{itemize}
If also $\Sigma\der{\lgc{L}} \f$ implies $\Sigma'\der{\lgc{L}} \f$ for some finite $\Sigma'\subseteq\Sigma$, then ${\der{\lgc{L}}}$ is called {\em finitary}.

An ordered pair $\lgc{L} = \langle \lang,\der{\lgc{L}}\rangle$, where $\der{\lgc{L}}$ is a substitution-invariant consequence relation  over a propositional language $\lang$, is called a {\em logic} over $\lang$. We call another logic $\lgc{L}'$  an {\em extension} of $\lgc{L}$ if ${\der{\lgc{L}}}\subseteq{\der{\lgc{L}'}}$. Given $X \subseteq \fm_\lang$, we also call the smallest extension of $\lgc{L}$ that includes $X$ an {\em axiomatic extension} of $\lgc{L}$ and  denote it by $\lgc{L} \oplus X$.

Let us consider a propositional language $\lang_m$ with binary connectives $\to,\cdot$ and constants $\ut,\nf$,  defining $\lnot\f := \f\to\nf$,  $\f + \ps := \lnot \f \to \ps$, and, inductively,  $0\f := \nf$, $\f^0 := \ut$, $(n+1)\f = n\f + \f$, and $\f^{n+1} = \f^n \cdot \f$ for  $n \in \N$. We shorten $\fm_{\lang_m}$ to $\fm_m$ and call its members {\em multiplicative} formulas. {\em Multiplicative linear logic} $\lgc{MLL}$ over $\lang_m$ can be defined via derivability in the axiom system:
\[
\begin{array}{c}
\begin{array}{lcl}
(\f  \to \ps) \to ((\ps \to \x) \to (\f \to \x)) & & (\f \to (\ps \to \x)) \to ((\f \cdot \ps) \to \x)\\[.025in]
(\f \to (\ps \to \x)) \to (\ps \to (\f \to \x))  & &  \f \to (\ps \to (\f \cdot \ps)) \\[.025in]
\f \to \f 						  & & \f \to (\ut \to \f)\\[.025in]
\lnot\lnot\f \to \f					  & & \ut\\[.025in]
\end{array}\\
\infer[(\text{mp})]{\ps}{\f & \f \to \ps} 
\end{array}
\]
The following useful deduction theorem is proved by an easy induction on the height of a derivation in an axiomatic extension of $\lgc{MLL}$.

\begin{Lemma}[cf.~\cite{Avr88}]\label{lem:deductiontheorem}
Let $\lgc{L}$ be an axiomatic extension of $\lgc{MLL}$. Then for any $\Si \cup \{\f,\ps\}\subseteq\fm_m$,
\[
\Si \cup \{\f\} \der{\lgc{L}} \ps
\iff
\Si  \der{\lgc{L}} \f^n \to \ps \, \text{ for some }n \in \N.
\]
\end{Lemma}

It will also be useful later to consider  the logic $\lgc{MLL^u}$ defined over $\lang_m$ by the axiom system for $\lgc{MLL}$ extended with the ``unperforated" rule schema
\[
\infer[(\text{u}_n) \enspace (n \in \N^+)]{\f}{n\f}
\]
Let $\lang_\ell$ be the propositional language with connectives $\land$, $\lor$, $\cdot$, $\to$, $\ut$, and $\nf$, shortening $\fm_{\lang_\ell}$ to $\fm_\ell$. {\em Multiplicative additive linear logic without additive constants} $\lgc{MALL^-}$  over $\lang_\ell$ can be defined via the axiom system for $\lgc{MLL}$ extended with the axiom and rule schema
\[
\begin{array}{c}
\begin{array}{lcl}
(\f \land \ps) \to \f & & \f \to (\f \lor \ps)\\[.025in]
(\f \land \ps) \to \ps & & \ps \to (\f \lor \ps)\\[.025in]
((\f \to \ps) \land (\f \to \x)) \to (\f \to (\ps \land \x)) & & ((\f \to \x) \land (\ps \to \x)) \to ((\f \lor \ps) \to \x)\\[.025in]
\end{array}\\
\infer[(\text{adj})]{\f \land \ps}{\f & \ps} 
\end{array}
\]
Appropriate algebraic semantics for $\lgc{MALL^-}$ and other substructural logics are provided by classes of residuated lattices~\cite{BT03,JT02,GJKO07}. An {\em  involutive commutative residuated lattice} is an algebraic structure  $\m{A}=\langle A,\land,\lor,\cdot,\to,\ut,\nf \rangle$ such that $\langle A,\land,\lor \rangle$ is a lattice (where $a \le b\, :\Longleftrightarrow\, a \land b = a$), $\langle A,\cdot,\ut \rangle$ is a monoid, $\lnot \lnot a = a$ for all $a \in A$, and $\to$ is the  residual  of $\cdot$, i.e., $b \le a \to c \iff  a \cdot b \le c$\, for all $a,b,c\in A$. It is easily shown (see, e.g.,~\cite{GJKO07}) that the class of all involutive commutative residuated lattices can be defined by equations and hence forms a variety that we denote by  $\InCRL$.

Let $\vty{K}$ be any class of involutive commutative residuated lattices. We define for $\Si\cup\{\f\}\subseteq \fm_\ell$,
\[
\Si \mdl{\vty{K}} \f\,  \defiff 
\begin{array}{c}
 \text{for any $\m{A}\in\vty{K}$ and homomorphism $e \colon \fma_\ell \to \m{A}$,}\\[.025in]
	 \ut \le e(\ps) \text{ for all } \ps \in \Si \ \Longrightarrow \ \ut \le e(\f).
\end{array}
\]
It is easily checked that $\mdl{\vty{K}}$ is a substitution-invariant consequence relation over $\lang_\ell$; moreover, if $\vty{K}$ is a variety (equational class), this consequence relation will be finitary (see,~e.g.,~\cite{MMT14}).

For any logic $\lgc{L} = \lgc{MALL^-} \oplus \mathcal{A}$ for some $\mathcal{A} \subseteq \fm_\ell$, we obtain a variety
\[
\vty{V}_\lgc{L} := \{\m{A}\in \InCRL \mid\, \mdl{\m{A}} \ps \text{ for all }\ps\in\mathcal{A}\}.
\]
The following algebraic completeness theorem is then standard.

\begin{Proposition}[cf.~\cite{GJKO07}]\label{prop:completeness}
If $\lgc{L} = \lgc{MALL^-} \oplus \mathcal{A}$ for some $\mathcal{A} \subseteq \fm_\ell$, then for all  $\Si\cup\{\f\}\subseteq \fm_\ell$,
\[
\Si\der{\lgc{L}} \f \iff \Si \mdl{\vty{V}_\lgc{L}}\f.
\]
\end{Proposition}

The {\em multiplicative fragment} of an extension $\lgc{L}$ of $\lgc{MALL^-}$ is the logic $\lgc{L_m}$  defined over $\lang_m$ with  ${\der{\lgc{L_m}}} := {\der{\lgc{L}}} \mathrel{\cap} (\mathcal{P}(\fm_m) \times \fm_m)$.  In order to reduce consequence in $\lgc{L}$ to consequence in  $\lgc{L_m}$, we require distributivity properties that are satisfied in particular when $\lgc{L}$ is complete with respect to a class of totally ordered algebras. We therefore  consider {\em involutive uninorm logic  without additive constants} (see~\cite{MM07}), which may be defined as
\[
\lgc{IUL^-} := \lgc{MALL^-} \oplus \{((p \to q) \land \ut) \lor ((q \to p) \land\ut)\}.
\]
For any variety $\vty{V}$ of involutive commutative residuated lattices, let us denote the class of totally ordered members of $\vty{V}$ by $\vty{V}^c$. For axiomatic extensions of $\lgc{IUL^-}$, we obtain the following more specialized completeness result.

\begin{Proposition}[cf.~\cite{MM07,GJKO07}]
If $\lgc{L} = \lgc{IUL^-} \oplus \mathcal{A}$ for some $\mathcal{A} \subseteq \fm_\ell$, then for all $\Si\cup\{\f\}\subseteq \fm_\ell$,
\[
\Si\der{\lgc{L}}\f \iff \Si \mdl{\vty{V}^c_\lgc{L}} \f.
\]
\end{Proposition}

Using the previous proposition, it is straightforward to prove that in any axiomatic extension $\lgc{L}$ of $\lgc{IUL^-}$, each formula $\f \in \fm_\ell$ is equivalent both to a conjunction of disjunctions of multiplicative formulas and a disjunction of conjunctions of multiplicative formulas. It is also straightforward to establish the following equivalences:
\begin{align*}
\Si \der{\lgc{L}} \f_1 \land \f_2 & \iff  \Si \der{\lgc{L}} \f_1\, \text{ and }\, \Si \der{\lgc{L}}  \f_2\\
\Si \cup \{\ps_1 \land \ps_2\}  \der{\lgc{L}} \f & \iff \Si \cup \{\ps_1,\ps_2\}  \der{\lgc{L}} \f\\
\Si \cup \{\ps_1 \lor \ps_2\}  \der{\lgc{L}} \f & \iff \Si \cup \{\ps_1\}  \der{\lgc{L}} \f \, \text{ and }\, \Si \cup \{\ps_2\}  \der{\lgc{L}} \f.
\end{align*}
Consequence in $\lgc{L}$ can therefore be reduced to consequences of the form $\Si \der{\lgc{L}} \f_1 \lor\ldots\lor\f_n$ where $\Si \cup \{\f_1,\ldots,\f_n\} \subseteq \fm_m$.

The following axiomatic extensions of $\lgc{IUL^-}$ will be of particular interest in this paper:
\begin{align*}
\lgc{A}			\ &  :=\ \lgc{IUL^-}\,\oplus\, \{(p \to p) \to \nf, \, \nf \to \ut\}\\
\lgc{RM^t}			\ &  :=\ \lgc{IUL^-}\,\oplus\, \{p \to (p + p), (p+p)\to p\} \\
\lgc{IUML^-}		\ &  :=\ \lgc{RM^t}\,\oplus\, \{\ut\to\nf\}\\
\lgc{BIUL^-}		\ &  :=\ \lgc{IUL^-}\,\oplus\, \{np \to p^n, p^n\to np  \mid n \in \N\}.
\end{align*}
The varieties $\vty{V}_\lgc{A}$, $\vty{V}_\lgc{RM^t}$, and $\vty{V}_\lgc{IUML^-}$ are term-equivalent to lattice-ordered abelian groups, Sugihara monoids, and odd Sugihara monoids, respectively, while $\vty{V}_\lgc{BIUL^-}$ (where ``B" stands for ``balanced") consists of involutive commutative residuated lattices satisfying $x^n \eq nx$ for all $n \in \N$.


\section{Theorems of Alternatives}\label{sec:alternatives}

We will say that an extension $\lgc{L}$ of $\lgc{IUL^-}$ admits a {\em theorem of alternatives} if for any multiplicative formulas $\Si \cup \{\f_1,\ldots,\f_n\} \subseteq \fm_m$,
\[
\Si \der{\lgc{L}} \f_1 \lor \ldots \lor \f_n \iff \Si \der{\lgc{L}}  \la_1 \f_1 + \cdots +  \la_n \f_n\, \text{ for some $\la_1,\ldots,\la_n \in \N$ not all $0$.}
\]
Such logics must satisfy the law of excluded middle and a ``mingle" axiom.

\begin{Lemma}\label{lem:useful}
Let $\lgc{L}$ be an extension of $\lgc{IUL^-}$ that admits a theorem of alternatives. Then
\begin{enumerate}
\item[\rm (i)]  $\der{\lgc{L}} p\lor \lnot p$
\item[\rm (ii)]  $\der{\lgc{L}} \nf \to \ut$.
\end{enumerate}
\end{Lemma}
\begin{proof}
For \text{(i)}, it suffices to observe that  $\der{\lgc{L}} \lnot p + p$ and  hence, by the theorem of alternatives for $\lgc{L}$, also  $\der{\lgc{L}} p\lor \lnot p$. For \text{(ii)}, we note first that $\der{\lgc{IUL^-}} (\lnot p + \lnot p) \lor (p + p)$. Hence, by the theorem of alternatives for $\lgc{L}$, we obtain $\der{\lgc{L}} \la  (\lnot p + \lnot p) + \mu (p+p)$ for some $\la,\mu\in\N$ not both $0$. If $\mu \neq 0$, then substituting $\ut$ for $p$ yields $\der{\lgc{L}} \mu (\ut + \ut)$ and, by a further application of the theorem of alternatives, $\der{\lgc{L}} \ut + \ut$. Similarly, if $\la \neq 0$, then substituting $\nf$ for $p$ yields $\der{\lgc{L}} \la  (\ut + \ut)$ and, again by the theorem of alternatives,   $\der{\lgc{L}} \ut + \ut$. In both cases it follows that $\der{\lgc{L}} \nf \to \ut$.
\end{proof}

We therefore define $\lgc{IUL^\star} = \lgc{IUL^-} \oplus \{p\lor \lnot p,\, \nf \to \ut\}$ and note that one direction of the theorem of alternatives holds for any extension of this logic.

\begin{Lemma}\label{lem:onedirection}
For any extension $\lgc{L}$ of $\lgc{IUL^\star}$ and $\Si \cup \{\f_1,\ldots,\f_n\} \subseteq \fm_m$,
\[
\Si \der{\lgc{L}}  \la_1 \f_1 + \cdots +  \la_n \f_n\, \text{ for some $\la_1,\ldots,\la_n \in \N$ not all $0$} \enspace\Longrightarrow\enspace \Si \der{\lgc{L}} \f_1 \lor \ldots \lor \f_n.
\]
\end{Lemma}
\begin{proof}
Observe that for any  $\Si \cup \{\f,\ps,\x\} \subseteq \fm_\ell$, if $\Si \der{\lgc{L}} (\f + \ps) \lor \x$, then $\Si \der{\lgc{L}} (\lnot \f \to \ps) \lor \x$ and, since $\der{\lgc{L}} \f \lor \lnot \f$, also $\Si \der{\lgc{L}} \f \lor \ps \lor \x$. The claim  follows by a simple inductive argument.
\end{proof}

We now establish a sufficient condition for extensions of $\lgc{IUL^\star}$ that are axiomatized via additional multiplicative formulas by considering  corresponding axiomatic extensions of $\lgc{MLL}$. Since $\nf \to \ut$ is derivable in the multiplicative fragment of any extension of $\lgc{IUL^\star}$, we let $\lgc{MLL_0} := \lgc{MLL} \oplus \{\nf \to \ut\}$ and $\lgc{MLL_0^u} :=  \lgc{MLL^u} \oplus \{\nf \to \ut\}$.

\begin{Lemma}\label{lem:thmalts}
Let $\mathcal{A} \subseteq \fm_m$ and  suppose that for any $\Si \cup\{\f\}\subseteq \fm_m$,
\[
\Si \der{\lgc{IUL^\star}\oplus\mathcal{A}} \f \iff \Si  \der{\lgc{MLL_0}\oplus\mathcal{A}} \la \f \ \mbox{ for some }\la \in \N^+.
\]
Then $\lgc{IUL^\star}\oplus\mathcal{A}$ admits a theorem of alternatives and its multiplicative fragment  is $\lgc{MLL_0^u}\oplus \mathcal{A}$.  
\end{Lemma}
\begin{proof}
By Lemma~\ref{lem:onedirection}, it suffices to prove the left-to-right direction of the theorem of alternatives for $\lgc{IUL^\star}\oplus\mathcal{A}$. Suppose that  $\Si \der{\lgc{IUL^\star}\oplus\mathcal{A}} \f_1 \lor \ldots \lor \f_n$ for some  $\Si \cup \{\f_1,\ldots,\f_n\} \subseteq \fm_m$ and let $p$ be a variable such that $p\not\in{\rm Var}(\Si \cup \{\f_1,\ldots,\f_n\})$. Observe that
\[
\{\f_1 \to p,\ldots,\f_n\to p,\f_1 \lor \ldots \lor \f_n\} \der{\lgc{IUL^\star}\oplus\mathcal{A}} p
\]
and hence
\[
\Si \cup \{\f_1 \to p,\ldots,\f_n\to p\} \der{\lgc{IUL^\star}\oplus\mathcal{A}}  p.
\]
By assumption, there exists $\la\in\N^+$ such that
\[
\Si \cup \{\f_1 \to p,\ldots,\f_n\to p\}  \der{\lgc{MLL_0}\oplus\mathcal{A}} \la p.
\]
But then, using Lemma~\ref{lem:deductiontheorem}, there exist $\la_1,\ldots,\la_n \in \N$ such that
\[
\Si \der{\lgc{MLL_0}\oplus\mathcal{A}}  (\f_1 \to p)^{\la_1} \to \ldots \to (\f_n \to p)^{\la_n}   \to \la p.
\]
If all the  $\la_1,\ldots,\la_n$ are $0$, then we can substitute $p$ with $\f_1$ and obtain $\Si \der{\lgc{MLL_0}\oplus\mathcal{A}}  \la \f_1$. Otherwise, we substitute $p$ with $0$ and obtain
\[
\Si  \der{\lgc{MLL_0}\oplus\mathcal{A}}  \la_1 \f_1 + \cdots +  \la_n \f_n.
\]
So clearly also $\Si  \der{\lgc{IUL^\star}\oplus\mathcal{A}}  \la_1 \f_1 + \cdots +  \la_n \f_n$. 

Finally, it follows directly from the assumption and the fact that $\{\la \f\} \der{\lgc{IUL^\star}} \f$ that $\lgc{MLL_0^u}\oplus \mathcal{A}$ is the multiplicative fragment of $\lgc{IUL^\star}\oplus\mathcal{A}$. 
\end{proof}

We are now able to formulate a sufficient condition for admitting a theorem of alternatives for logics axiomatized relative to $\lgc{IUL^\star}$ by multiplicative formulas.

\begin{Theorem} \label{t:toa}
Let $\mathcal{A}\subseteq \fm_m$ and  suppose that for some $n_0 \in \N$, whenever $n \ge n_0$, there exist $m\in\N^+,k\in\N$ such that $\der{\lgc{MLL_0}\oplus\mathcal{A}} (np)^k\to m(p^n)$. Then $\lgc{IUL^\star}\oplus\mathcal{A}$ admits a theorem of alternatives and its multiplicative fragment  is $\lgc{MLL_0^u}\oplus \mathcal{A}$.  
\end{Theorem}
\begin{proof}
Assume that for some $n_0 \in \N$, whenever $n \ge n_0$, there exist $m\in\N^+,k\in\N$ such that $\der{\lgc{IUL^\star}\oplus\mathcal{A}} (np)^k\to m(p^n)$.  By Lemma~\ref{lem:thmalts}, to show that  $\lgc{IUL^\star}\oplus\mathcal{A}$ admits a theorem of alternatives and its multiplicative fragment  is $\lgc{MLL_0^u}\oplus \mathcal{A}$, it suffices to prove that for any $\Si \cup\{\f\}\subseteq \fm_m$,
\[
\Si \der{\lgc{IUL^\star}\oplus\mathcal{A}} \f \iff \Si  \der{\lgc{MLL_0}\oplus\mathcal{A}} \la \f \ \mbox{ for some }\la \in \N^+.
\]
Suppose first that $ \Si  \der{\lgc{MLL_0}\oplus\mathcal{A}} \la \f$\, for some $\la\in \N^+$. Then also $\Si\der{\lgc{IUL^\star}\oplus\mathcal{A}}  \la \f$ and, since $\{\la \f\} \der{\lgc{IUL^\star}\oplus\mathcal{A}} \f$, it follows that  $\Si \der{\lgc{IUL^\star}\oplus\mathcal{A}}  \f$.  To prove the converse, we assume contrapositively that $\Si_0 := \Si$ satisfies
\[
(\star) \enspace \Si_0 \not \der{\lgc{MLL_0}\oplus\mathcal{A}}  \la \f \ \text{ for all }\la\in\N^+.
\]
We enumerate $\fm_m$ as $(\ps_i)_{i\in\N}$. Suppose now that $\Si_N$ for some $N\in\N$ contains $\ps_i$ or $\lnot \ps_i$ for all $i < N$ and satisfies $(\star)$. Consider $\ps_N$ and suppose for a contradiction that for some $\la,\mu \in\N^+$,
\[
\Si_N \cup \{\ps_N\} \der{\lgc{MLL_0}\oplus\mathcal{A}}  \la \f \enspace \text{and} \enspace
\Si_N \cup \{\lnot \ps_N\} \der{\lgc{MLL_0}\oplus\mathcal{A}}  \mu \f.
\]
By Lemma~\ref{lem:deductiontheorem}, there exist $r,s \in \N$ such that
\[
\Si_N\der{\lgc{MLL_0}\oplus\mathcal{A}}   (\ps_N)^r \to\la \f \enspace \text{and} \enspace
\Si_N  \der{\lgc{MLL_0}\oplus\mathcal{A}}  (\lnot \ps_N)^s \to \mu \f.
\]
Using the fact that $\der{\lgc{MLL_0}} \nf \to 1$, it  follows easily that also
\[
\text{(i)} \enspace \Si_N\der{\lgc{MLL_0}\oplus\mathcal{A}}  (\ps_N)^{rs} \to s \la \f  \enspace \text{and} \enspace
\text{(ii)} \enspace  \Si_N\der{\lgc{MLL_0}\oplus\mathcal{A}}    (\lnot\ps_N)^{rs} \to  r \mu \f,
\]
where \text{(ii)} can be rewritten, more conveniently, as
\[
\text{(ii')} \enspace  \Si_N\der{\lgc{MLL_0}\oplus\mathcal{A}}    rs\ps_N  + r \mu \f.
\]
By assumption, with $n=rst$ for some large $t \in \N^+$, there exist $m \in  \N^+,k\in\N$  such that
\[
\der{\lgc{MLL_0}\oplus\mathcal{A}}  (rst (\ps_N))^k \to m(\ps_N)^{rst}.
\]
Observe also that, using \text{(i)} and {\rm (ii')},
\[
\Si_N\der{\lgc{MLL_0}\oplus\mathcal{A}}  m(\ps_N)^{rst} \to mst\la \f\enspace \text{and} \enspace
\Si_N\der{\lgc{MLL_0}\oplus\mathcal{A}}  (rst)\ps_N  + (rt \mu) \f.
\]
Hence $\Si_N \der{\lgc{MLL_0}\oplus\mathcal{A}} (rst (\ps_N))^k \to mst\la \f$,  yielding 
\[
\Si_N \der{\lgc{MLL_0}\oplus\mathcal{A}} (mst\la + krt\mu) \f,
\]
which contradicts the assumption that $\Si_N$ satisfies $(\star)$. So $\Si_N$ can be extended with either $\ps_N$ or $\lnot\ps_N$ to obtain $\Si_{N+1}\subseteq \fm_m$ that  satisfies $(\star)$. We then let
\[
\Si^* := \bigcup_{i \in \N} \Si_i,
\]
noting that, by a simple argument using the fact that $\der{\lgc{MLL_0}\oplus\mathcal{A}}$ is finitary, $\Si^*$ also satisfies $(\star)$. 

Next, we define a binary relation $\Theta$ on $\fm_m$ by
\[
\ps \mathrel{\Theta} \x \defiff \Si^* \der{\lgc{MLL_0}\oplus\mathcal{A}} \ps \to \x \ \text{ and }\ \Si^* \der{\lgc{MLL_0}\oplus\mathcal{A}} \x \to \ps.
\]
It is then straightforward to show that $\Theta$ is in fact a congruence on $\m{Fm_m}$ and hence that the set of equivalence classes $\fm_m / \Theta = \{[\ps] \mid \ps\in \fm_m\}$, where $[\ps] = \{\x\in\fm_m \mid \ps\mathrel{\Theta}\x\}$, can be equipped with well-defined binary operations $\cdot$ and $\to$ and constants $[\ut]$ and $[\nf]$. Now define also
\[
[\ps] \le [\x] \defiff  \Si^* \der{\lgc{MLL_0}\oplus\mathcal{A}} \ps \to \x.
\]
This is a total order by construction and hence we can equip $\langle\fm_m / \Theta,\cdot,\to,[\ut],[\nf]\rangle$ also with meet and join operations $\land$ and $\lor$. It is then straightforward to show that the resulting algebra belongs to $\InCRL$ and satisfies each member of $\mathcal{A} \cup  \{p\lor \lnot p,\, \nf \to \ut\}$. Finally, we consider a homomorphism $e$ mapping each formula $\x$ to its equivalence class $[\x]$, obtaining $[\ut] \le e(\ps)$ for all $\ps\in\Si^\star$ and $[\ut] \not\le e(\f)$. Hence $\Si \not \der{\lgc{IUL^\star}\oplus\mathcal{A}} \f$ as required.
\end{proof} 

Clearly, the logics $\lgc{BIUL^-}$,  $\lgc{A}$, $\lgc{RM^t}$, and $\lgc{IUML^-}$ defined in Section~\ref{sec:prelim} satisfy the condition of these previous theorem and admit a theorem of alternatives. More generally, we obtain the following result for extensions of $\lgc{BIUL^-}$.

\begin{Corollary}
For any $\mathcal{A}\subseteq \fm_m$, the logic $\lgc{BIUL^-}\oplus\mathcal{A}$ admits a theorem of alternatives and its multiplicative fragment  is $\lgc{MLL_0^u}\oplus \{np \to p^n, p^n \to np \mid n \in \N\} \oplus\mathcal{A}$.  
\end{Corollary}

Note that in the statement of Theorem~\ref{t:toa}, the $m\in\N^+,k\in\N$ satisfying the condition $\der{\lgc{MLL_0}\oplus\mathcal{A}} (np)^k\to m(p^n)$ depend in general on the particular $n \ge n_0$. If the logic proves a knotted axiom of the form $p^t \to p^{t+1}$ ($t\in\N^+$), however, these parameters can be fixed.

\begin{Corollary}
For any $\mathcal{A}\subseteq \fm_m$ and  $k, m, r, s, t\in \N^+$, with $r,s \geq t$, the logic $\lgc{IUL^\star}\oplus \mathcal{A}\cup\{p^t \to p^{t+1}, (r p)^k \to m (p^s)\}$ admits a theorem of alternatives and its multiplicative fragment  is $\lgc{MLL_0^u}\oplus\mathcal{A}\cup\{p^t \to p^{t+1}, (r p)^k \to m (p^s)\}$.  
\end{Corollary}

For example, the logic $\lgc{IUL^\star}\oplus  \{p^2 \to p^3,(4 p)^5 \to 6 (p^7)\}$ has a theorem of alternatives.  More generally, for any  $\mathcal{A}\subseteq \fm_m$ and $t,u,r_0, k_0, m_0, s_0, \ldots, r_{u-1}, k_{u-1}, m_{u-1}, s_{u-1} \in \N^+$, where all the $r_i$ and $s_i$ are congruent to $i$ modulo $u$ and greater or equal to $t$, the logic  $\lgc{IUL^\star}\oplus \mathcal{A}\cup \{p^t \to p^{t+u}\} \cup \{ (r_i p)^{k_i} \to m_i (p^{s_i}) \mid 0 \leq i < u\}$ admits a theorem of alternatives. To see this, we apply Theorem~\ref{t:toa} with $n_0=\max(\{r_i \mid i<u\} \cup \{s_i \mid i<u\})$. For each  $n\geq n_0$, we let $i$ be the remainder of dividing $n$ by $u$ and choose $k=k_i$, $m=m_i$.  That $(np)^k \to m p^n$ is derivable in this logic can be shown by reasoning that in the corresponding variety of residuated lattices, $(n p)^{k_i} \leq (r_i p)^{k_i} \leq m_i (p^{s_i}) \leq m_i (p^n)$. (The last inequality follows from repeated applications of $p^t \leq p^{t+u}$ and the first by repeated applications of $(t+u)p \leq t p$, which follows from $p^t \leq p^{t+u}$ and involutivity.) 

Let us note that in the special cases of $\lgc{A}$, $\lgc{RM^t}$, and $\lgc{IUML^-}$, the theorem of alternatives can be established {\`a} la Avron~\cite{Avr87} using either the completeness of the logic and its multiplicative fragment with respect to certain algebras or a hypersequent calculus that admits cut elimination. However, in the case of $\lgc{BIUL^-}$ and other logics covered by the above results, suitable algebras and hypersequent calculi are not available, so these methods cannot be followed. What can be said is that if an extension  of  $\lgc{IUL^\star}$ admits a theorem of alternatives, then any analytic calculus for its multiplicative fragment can be extended to an analytic calculus for the full logic using versions of the mix and split rules.


\section{Interpolation}\label{sec:interpolation}

A logic $\lgc{L}$ over a propositional language $\lang$ is said to have the {\em deductive interpolation property} if for any finite $\Si\cup\{\f\} \subseteq \fm_\lang$ satisfying $\Si\der{\lgc{L}}\f$, there exists $\Pi\subseteq \fm_\lang$ with ${\rm Var}(\Pi) \subseteq {\rm Var}(\Si) \cap {\rm Var}(\f)$ such that  $\Pi\der{\lgc{L}}\f$ and $\Si\der{\lgc{L}}\ps$ for all $\ps\in\Pi$. It is easily shown (see, e.g.,~\cite{vGMT17}) that this is equivalent to the following condition:

\begin{itemize}
\item[($\dagger$)]
For any finite $\Si \subseteq \fm_\lang$ and $X \subseteq {\rm Var}(\Si)$, there exists  $\Pi\subseteq \fm_\lang$ with  ${\rm Var}(\Pi) \subseteq X$ such that for any $\f\in\fm_\lang$ satisfying $ {\rm Var}(\Si) \cap {\rm Var}(\f) \subseteq X$,
\[
\Si\der{\lgc{L}}\f \iff \Pi\der{\lgc{L}}\f.
\]
\end{itemize}

\noindent
If $\Pi$ in ($\dagger$) can always be finite, then $\lgc{L}$ is said to have the {\em right uniform deductive interpolation property}. If $\Pi$ can always be finite, but ($\dagger$) is restricted to formulas $\f\in\fm_\lang$ with ${\rm Var}(\f) \subseteq X$, then  $\lgc{L}$ is said to be {\em coherent}. It is proved in~\cite{KM19} that $\lgc{L}$ has the right uniform deductive interpolation property if and only if it has the deductive interpolation property and is coherent.  

Recall that the {\em multiplicative fragment} of an extension $\lgc{L}$ of $\lgc{MALL^-}$ is the logic $\lgc{L_m}$  defined over $\lang_m$ with  consequence relation ${\der{\lgc{L_m}}} := {\der{\lgc{L}}} \mathrel{\cap} (\mathcal{P}(\fm_m) \times \fm_m)$. We show now that an extension of $\lgc{IUL^-}$ that admits a theorem of alternatives inherits deductive interpolation and coherence properties from its multiplicative fragment. 

\begin{Theorem}
Let $\lgc{L}$ be an extension of $\lgc{IUL^-}$ that admits a theorem of alternatives.
\begin{itemize}
\item[\rm (a)] If $\lgc{L_m}$ has the deductive interpolation property, then so does $\lgc{L}$.
\item[\rm (b)] If $\lgc{L_m}$ is coherent, then so is $\lgc{L}$.
\item[\rm (c)] If $\lgc{L_m}$ has the right uniform deductive interpolation property, then so does $\lgc{L}$.
 \end{itemize}
\end{Theorem}
\begin{proof}
Suppose for (a) that  $\lgc{L_m}$ has the deductive interpolation property. We consider first any finite $\Si \subseteq \fm_m$ and $X \subseteq {\rm Var}(\Si)$. By assumption, there exists $\Pi\subseteq \fm_m$ such that  ${\rm Var}(\Pi) \subseteq X$ and for any $\f\in\fm_m$ satisfying $ {\rm Var}(\Si) \cap {\rm Var}(\f) \subseteq X$,
\[
\Si\der{\lgc{L}}\f \iff \Pi\der{\lgc{L}}\f.
\]
Hence also for any $\f_1,\ldots,\f_n\in\fm_m$ satisfying ${\rm Var}(\Si) \cap {\rm Var}(\{\f_1,\ldots,\f_n\}) \subseteq X$, by the theorem of alternatives,
\begin{align*}
\Si \der{\lgc{L}} \f_1 \lor \ldots \lor \f_n & \iff \Si \der{\lgc{L}}  \la_1 \f_1 + \cdots +  \la_n \f_n\, \text{ for some $\la_1,\ldots,\la_n \in \N$ not all $0$}\\
& \iff  \Pi  \der{\lgc{L}}  \la_1 \f_1 + \cdots +  \la_n \f_n\, \text{ for some $\la_1,\ldots,\la_n \in \N$ not all $0$}\\
& \iff \Pi  \der{\lgc{L}}  \f_1 \lor \ldots \lor \f_n.
\end{align*}
Moreover, recalling that every $\f\in\fm_\ell$ is equivalent in $\lgc{L}$ to a conjunction of disjunction of formulas in $\fm_m$ and that for any $\De \cup \{\ps_1,\ps_2\}\subseteq\fm_\ell$,
\[
\De \der{\lgc{L}} \ps_1 \land \ps_2 \iff  \De \der{\lgc{L}} \ps_1\, \text{ and }\, \De \der{\lgc{L}}  \ps_2,
\]
it follows that for any $\f\in\fm_\ell$  satisfying $ {\rm Var}(\Si) \cap {\rm Var}(\f) \subseteq X$,
\[
\Si\der{\lgc{L}}\f \iff \Pi\der{\lgc{L}}\f.
\]
Now consider any finite $\Si \subseteq \fm_\ell$ and $X \subseteq {\rm Var}(\Si)$. Since for any $\De \cup \{\ps_1,\ps_2\}\subseteq\fm_\ell$,
\[
\De \cup \{\ps_1 \land \ps_2\}  \der{\lgc{L}} \f  \iff \De \cup \{\ps_1,\ps_2\}  \der{\lgc{L}} \f,
\]
we may assume that $\Si$ consists of disjunctions of formulas in $\fm_m$. Suppose that $\Si = \Si' \cup \{\ps_1 \lor \ps_2\}$ and there exist $\Pi_1 \cup \Pi_2 \subseteq \fm_\ell$ such that  ${\rm Var}(\Pi_1 \cup \Pi_2) \subseteq X$ and for any $\f\in\fm_\ell$ satisfying ${\rm Var}(\Si) \cap {\rm Var}(\f) \subseteq X$,
\[
\Si' \cup \{\ps_1\} \der{\lgc{L}}\f \iff \Pi_1\der{\lgc{L}}\f \quad\text{and}\quad \Si' \cup \{\ps_2\} \der{\lgc{L}}\f \iff \Pi_2\der{\lgc{L}}\f.
\]
We define
\[
\Pi := \{(\x_1 \land \ldots \land \x_n) \lor (\x'_1\land \ldots \land \x'_m) \mid \x_1,\ldots,\x_n \in \Pi_1,\, \x'_1,\ldots,\x'_m \in \Pi_2\}.
\]
Then ${\rm Var}(\Pi) \subseteq X$ and for any $\f\in\fm_\ell$ satisfying ${\rm Var}(\Si) \cap {\rm Var}(\f) \subseteq X$,
\begin{align*}
\Si \der{\lgc{L}} \f  & \iff\Si' \cup\{\ps_1 \lor \ps_2\} \der{\lgc{L}} \f\\
& \iff \Si' \cup\{\ps_1\} \der{\lgc{L}} \f \enspace\text{and}\enspace \Si' \cup\{\ps_2\} \der{\lgc{L}} \f\\
& \iff \Pi_1 \der{\lgc{L}} \f \enspace\text{and}\enspace \Pi_2 \der{\lgc{L}} \f\\
& \iff \Pi \der{\lgc{L}} \f.
\end{align*}
Hence it follows by induction on the number of occurrences of $\lor$ in $\Si$ that there exists $\Pi\subseteq \fm_\lang$ such that  ${\rm Var}(\Pi) \subseteq X$ and for any $\f\in\fm_\lang$ satisfying $ {\rm Var}(\Si) \cap {\rm Var}(\f) \subseteq X$,
\[
\Si\der{\lgc{L}}\f \iff \Pi\der{\lgc{L}}\f.
\]
For (b) and (c), we just note that if  $\lgc{L_m}$ is  coherent or has the right uniform deductive interpolation property, then the construction described above will also yield a set $\Pi$ that is  finite and hence $\lgc{L}$ will be coherent or have  the right uniform deductive interpolation property, respectively.
\end{proof}

For example, deductive interpolation for the multiplicative fragment of $\lgc{RM^t}$ can be established proof theoretically via a Maehara lemma argument for the sequent calculus for this fragment defined in~\cite{Avr87}. This yields a further proof of the fact that $\lgc{RM^t}$ has the deductive interpolation property, first proved in~\cite{Mey80} (see also~\cite{Avr86,MM12}). Indeed, since the variety of Sugihara algebras corresponding to $\lgc{RM^t}$ is locally finite, this logic is coherent and has the right uniform deductive interpolation property. Similarly, it can be shown semantically that the multiplicative fragment of $\lgc{A}$ has  the right uniform deductive interpolation property and hence that the same holds for the full logic (see~\cite{MMT14} for a proof that proceeds along these lines without mentioning a theorem of alternatives).


\section{Density}\label{sec:density}

Recall from Section~\ref{sec:prelim} that any axiomatic extension $\lgc{L}$ of $\lgc{IUL^-}$ is complete both with respect to a variety $\vty{V}_\lgc{L}$ of involutive commutative residuated lattices and to the class $\vty{V}^c_\lgc{L}$ of totally ordered members of $\vty{V}_\lgc{L}$. In this section we show that if $\lgc{L}$ has a theorem of alternatives and $\der{\lgc{L}} \ut \to \nf$, then $\lgc{L}$ is also complete with respect to the class $\vty{V}^d_\lgc{L}$ of {\em dense} totally ordered members of $\vty{V}_\lgc{L}$. From an algebraic perspective, such a completeness result corresponds to  $\vty{V}_\lgc{L}$ being generated as a generalized quasivariety by the class $\vty{V}^d_\lgc{L}$ (i.e., $\vty{V}_\lgc{L} = \op{ISP}(\vty{V}^d_\lgc{L})$) or, equivalently, the property that each member of  $\vty{V}^c_\lgc{L}$ embeds into a member of  $\vty{V}^d_\lgc{L}$  (see~\cite{MT17} for details).

Let us say that  an extension $\lgc{L}$ of $\lgc{IUL^-}$ is {\em dense chain complete} if for any $\Si \cup \{\f\}\subseteq \fm_\ell$,
\[
\Si\der{\lgc{L}}\f \iff \Si \mdl{\vty{V}^d_\lgc{L}} \f.
\]
If $\lgc{L}$ has a theorem of alternatives and is dense chain complete, then $\der{\lgc{L}} \ut \to \nf$. Just consider any $\m{A}\in\vty{V}^d_\lgc{L}$ and observe that $\der{\lgc{L}} (\ut \to x) \lor (x \to \nf)$ by part (i) of Lemma~\ref{lem:useful}, so $\ut \le x$ or $x \le \nf$ for all $x \in A$, and, since $\m{A}$ is dense, $\ut = \nf$ and, by dense chain completeness,  $\der{\lgc{L}} \ut \to \nf$. It follows, for example, that $\lgc{RM^t}$ is not dense chain complete, although, as shown below (or see~\cite{MM07}), $\lgc{IUML^-} = \lgc{RM^t} \oplus \ut \to \nf$ does have this property.

One method for establishing dense chain completeness for a logic is to establish the admissibility of a certain ``density rule" (see~\cite{MM07,CM08}). We say that an extension $\lgc{L}$ of $\lgc{IUL^-}$ has the {\em density property} if for any $\Si \cup \{\f,\ps,\x\}\subseteq \fm_\ell$ and $p\not\in{\rm Var}(\Si \cup \{\f,\ps,\x\})$,
\[
\Si \der{\lgc{L}} (\f \to p) \lor (p \to \ps) \lor \x \iff \Si \der{\lgc{L}}    (\f \to \ps) \lor \x.
\]
We make use of the following result,  proved in a more general setting in~\cite{MM07} (see also~\cite{CM08,BT16,MT17,GH18}).

\begin{Theorem}[\cite{MM07}]\label{thm:densitycondition}
Any extension of $\lgc{IUL^-}$ that has the density property is dense chain complete.
\end{Theorem}

\begin{Theorem}
Any extension of $\lgc{IUL^-}\oplus\{\ut\to\nf\}$ that admits a theorem of alternatives is dense chain complete.
\end{Theorem}
\begin{proof}
Let $\lgc{L}$ be an extension of $\lgc{IUL^-}\oplus\{\ut\to\nf\}$ that admits a theorem of alternatives. By Theorem~\ref{thm:densitycondition}, it suffices to prove that $\lgc{L}$ has the density property. Suppose first that for some $\Si \cup \{\f,\ps,\x\}\subseteq \fm_m$ and $p\not\in{\rm Var}(\Si \cup \{\f,\ps,\x\})$,
\[
\Si \der{\lgc{L}} (\f \to p) \lor (p \to \ps) \lor\x.
\]
Since  $\lgc{L}$ admits a theorem of alternatives, there exist $\la,\mu,\ga\in \N$ not all $0$ such that
\[
\Si \der{\lgc{L}} \la(\f \to p) + \mu(p \to \ps) + \ga \x.
\]
Substituting $p$ with $\nf$, and separately all other variables with $\nf$,  yields the consequences
\[
\Si \der{\lgc{L}}  \la (\lnot \f) + \mu \ps + \ga \x \enspace \text{and} \enspace
\der{\lgc{L}} \la p + \mu (\lnot p).
\]
Multiplying the conclusion in the first consequence by $\la$ and substituting $p$ with $\f^\la$ in the second consequence produces the consequences
\[
\Si \der{\lgc{L}}  \la(\f^\la) \to (\la\mu \ps + \la\ga \x) \enspace \text{and} \enspace
\der{\lgc{L}} \f^{\la\mu} \to \la (\f^\la).
\]
By transitivity of implication, we obtain
\[
\Si \der{\lgc{L}}   \f^{\la\mu} \to (\la\mu \ps + \la\ga \x),
\]
 which can be rewritten as
\[
\Si \der{\lgc{L}} \la\mu (\f \to \ps) + \la\ga \x.
\]
By the theorem of alternatives again, $\Si \der{\lgc{L}}  (\f \to \ps) \lor \x$.

Now consider the more general case where $\Si \der{\lgc{L}} (\f \to p) \lor (p \to \ps) \lor\x$ for some $\Si \cup \{\f,\ps,\x\}\subseteq \fm_\ell$ and $p\not\in{\rm Var}(\Si \cup \{\f,\ps,\x\})$. If $\f,\ps,\x \in \fm_m$, then using the equivalences presented in Section~\ref{sec:prelim} and the multiplicative case just established, we obtain again $\Si \der{\lgc{L}}  (\f \to \ps) \lor \x$ as required. Otherwise, for $q,r,s \not\in\{p\} \cup {\rm Var}(\Si \cup \{\f,\ps,\x\})$, we obtain
\[
\Si \cup \{q \to \f, \ps \to r, \x \to s\} \der{\lgc{L}} (q \to p) \lor (p \to r) \lor s.
\]
But then also using the equivalences presented in Section~\ref{sec:prelim} and the multiplicative case just established, 
\[
\Si \cup \{q \to \f, \ps \to r, \x \to s\} \der{\lgc{L}} (q \to r) \lor s,
\]
and finally, substituting $\f$ for $q$, $\ps$ for $r$, and $\x$ for $s$ yields $\Si \der{\lgc{L}}  (\f \to \ps) \lor \x$.
\end{proof}

\begin{Corollary}
Let $\mathcal{A}\subseteq \fm_m$ and $\lgc{L} = \lgc{BIUL^-}\oplus \mathcal{A}$. Then $\lgc{L}$ is dense chain complete.
\end{Corollary}

Let us remark finally that dense chain completeness can be established for $\lgc{A}$ and $\lgc{IUML^-}$ via a direct semantic argument or proof-theoretically using an analytic hypersequent calculus as described in~\cite{MM07}, but these methods do not seem to be available for $\lgc{BIUL^-}$ or other logics admitting a theorem of alternatives. It may be hoped also that this new approach provides a first step towards addressing the open standard completeness problem for involutive uninorm logic posed in~\cite{MM07}, possibly by introducing a weaker theorem of alternatives property.\footnote{Note that a very different approach to tackling this problem, via representations of totally ordered involutive commutative residuated lattices using ordered groups, has been described recently in~\cite{Jen18}.}

\section*{Acknowledgements}
\noindent The research reported here was funded by the Swiss National Science Foundation (SNF) grant 200021$\_$184693 and the EU Horizon 2020 research and innovation programme under the Marie Sk{\l}odowska-Curie grant agreement No 689176.

\end{document}